\documentclass[12pt]{amsart}
\usepackage{amsmath,amssymb,amsfonts,amsthm,amsopn}
\usepackage{graphics}
\usepackage{srcltx}
\usepackage{color}

\usepackage{mathtools}
\newcommand{\tfs}{time-frequency shift}

\newtheorem{tm}{Theorem}[section]
\newtheorem{lemma}[tm]{Lemma}

\newtheorem{theorem}{Theorem}[section]
\newtheorem{corollary}[theorem]{Corollary}
\newtheorem{definition}[theorem]{Definition}

\newtheorem{proposition}[theorem]{Proposition}
\newtheorem{remark}[theorem]{Remark}

\newcommand{\beqa}{\begin{eqnarray*}}
\newcommand{\eeqa}{\end{eqnarray*}}
\newcommand{\tpsdo}{$\tau$-pseudo\-differential operator}
\DeclareMathOperator*{\supp}{supp}

\newcommand{\field}[1]{\mathbb{#1}}
\newcommand{\bR}{\field{R}}        
\newcommand{\bN}{\field{N}}        
\newcommand{\bZ}{\field{Z}}        
\newcommand{\opt }{\mathrm{Op}_{\tau}}
\newcommand{\opj }{\mathrm{Op}_{BJ}}

\def\Q{\mathcal{Q}}
\def\al{\alpha}                    
\def\be{\beta}

\def\la{\lambda}

\def\eps{\epsilon}

 \def\cF{\mathcal{F}}              
 \def\cS{\mathcal{S}}
 \def\cD{\mathcal{D}}

 \def\cG{\mathcal{G}}

 \def\cC{\mathcal{C}}

\def\a{\aleph}

\def\rd{\bR^d}

\def\rdd{{\bR^{2d}}}
\def\zdd{{\bZ^{2d}}}

\def\lrd{L^2(\rd)}

\def\zd{\bZ^d}

\def\intrd{\int_{\rd}}
\def\intrdd{\int_{\rdd}}

\def\R{\right)}

\def\<{\left<}
\def\>{\right>}

\def\mv1{M_v^1}

\def\mpq{M^{p,q}}

\def\phas{(x,\omega )}

\newcommand{\abs}[1]{\lvert#1\rvert}
\newcommand{\norm}[1]{\lVert#1\rVert}

\def\o{\omega}
\def\a{\alpha}

\def\ZZ{\mathbb{Z}}

\def\R{\mathbb{R}}
\def\Ren{\mathbb{R}^d}
\def\Renn{\mathbb{R}^{2d}}

\def\Sn2{S_{2}(L^{2}(\Ren))}
\def\S1{S_{1}(L^{2}(\Ren))}
\def\sig00{\sigma_{0,0}}

\def\la{\langle}
\def\ra{\rangle}

\renewcommand{\zeta}{u}

\begin{document}

\title[]{Characterization of smooth symbol classes by Gabor matrix decay}
\author{Federico Bastianoni}
\address{Depatment of Mathematical Sciences, Politecnico di Torino, corso
	Duca degli Abruzzi 24, 10129 Torino, Italy}
\email{federico.bastianoni@polito.it}

\author{Elena Cordero}
\address{Department of Mathematics,  University of Torino,
Via Carlo Alberto 10, 10123
Torino, Italy}
\email{elena.cordero@unito.it}

\keywords{Time-frequency analysis, modulation spaces,  Gabor matrix,  pseudodifferential operators, Gabor frames}

\subjclass[2010]{47G30,42B35,81S30}

\date{}

\begin{abstract} For $m\in\bR$ we consider  the symbol classes $S^m$, $m\in\bR$, consisting of smooth functions $\sigma$ on $\rdd$ such that $|\partial^\alpha \sigma(z)|\leq C_\alpha (1+|z|^2)^{m/2}$, $z\in\rdd$, and we show that can be characterized by an intersection of different types of modulation spaces. In the case $m=0$ we recapture the H\"{o}rmander class $S^0_{0,0}$ that can be obtained by intersection of suitable  Besov spaces as well. Such spaces contain the Shubin classes $\Gamma^m_\rho$, $0<\rho\leq1$, and can be viewed as their limit case $\rho=0$.  We exhibit almost diagonalization properties for the Gabor matrix  of $\tau$-pseudodifferential operators with symbols in such classes,   extending the characterization proved by Gr\"{o}chenig and Rzeszotnik in \cite{GR}. Finally, we compute the Gabor matrix of a Born-Jordan operator, which allows to prove new boundedness results for such operators. 
\end{abstract}

\maketitle

\section{Introduction and results}

Modulation spaces  were originally introduced by Feichtinger \cite{F1} in 1983   and have revealed to be very useful in many different frameworks, which include harmonic analysis, quantum mechanics, pseudodifferential and Fourier integral operators, partial differential equations (we refer the reader to Section $2$ for their definitions and main properties).

Several authors have studied inclusion relations of such spaces with other classical function spaces such as Besov, Triebel-Lizorkin Gelfand-Shilov spaces \cite{Guo-incl-mod-Triebel2017,sugimototomita,Toftweight2004,Wangbook2011}. In particular, when they are considered as symbol classes for pseudodifferential or Fourier integral operators, their relationship  with classical symbol spaces such as the H\"{o}rmander classes or the Shubin-Sobolev spaces has been investigated in many contributions (see e.g.,  \cite{BCG,Elena-book,NR,ToftquasiBanach2017} and the references therein). 

 In 1994 Sj\"ostrand \cite{wiener30} introduced the first symbol class via  time-frequency concentration on the phase-space, the Sj\"ostrand  class, which later revealed to be a type of modulation space. This rough symbol class have been inspired many works on pseudodifferential operators with symbols in modulation spaces (see, e.g., \cite{benyi,Bi2010,Elena-book,grochenig2,CharlyToft2011,Nenad2016,Nenad2018,PT-JFA2004,toft1,Toftweight2004} and the book \cite{Elena-book}). The contributions are so many that it is not possible to cite them all. 
 
 In \cite{Sjostrandfirst}  Sj\"ostrand   continued his study on pseudodifferential operators with rough symbols and he also considered the symbol class object of our study. Namely,  for $m\in\bR$, let us define    
\begin{equation}\label{Sm00}
S^m(\rdd)=\{\sigma\in\cC^\infty(\rdd):\, |\partial^\alpha\sigma(z)|\leq C_\alpha \la z\ra^m,\quad \a\in\mathbb{N}^{2d},\,\, z\in\rdd\},
\end{equation}
for the definition of $\la z\ra^m$ see \eqref{vs}. Notice that this is a special instance of the class $S(w)$ introduced in  \cite[Formula $(3.2)$]{Sjostrandfirst}.  

There were several papers/books in  the seventies and eighties where
this symbol class were considered. For example, the whole theory of
the Weyl calculus, e.g. in \cite{Bony1994} can be applied on this class.

Another work on pseudodifferential operators with symbols of the type above is due to 
Rochberg and K. Tachizawa \cite{rochberg}. Later, these classes were  considered as spaces for symbols of Fourier integral operators \cite[Remark 3.2]{fio1}.

For $m=0$ we recapture  the standard  H\"{o}rmander class $S^0_{0,0}(\rdd)$: pseudodifferential operators with these symbols are an algebra which is closed under inversion. This claim was originally proved by Beals in \cite{beals} and later recaptured by 
Gr{\"o}chenig and  Rzeszotnik in \cite{GR}, using time-frequency analysis; key tool was the almost diagonalization property of the related Gabor matrix.

We continue this spirit of investigation and present a characterization of pseudodifferential operators with symbols in $S^m(\rdd)$ in terms of the decay properties of the related Gabor matrix. Let us introduce the main features of this work.

For $\tau \in [0,1]$, the (cross-)$\tau $-Wigner distribution is the time-frequency representation defined by 
\begin{equation}
W_{\tau }(f,g)(x,\omega )=\int_{\mathbb{R}^{d}}e^{-2\pi iy\omega }f(x+\tau y)%
\overline{g(x-(1-\tau )y)}\,dy,\quad f,g\in \mathcal{S}(\mathbb{R}^{d}),
\label{tauwig}
\end{equation}
cf. \cite{Janssen1985}. Given any tempered distribution  $\sigma \in\cS'(\rdd)$, the \tpsdo\,  $ \opt  (\sigma)$ can be introduced weakly  as
\begin{equation}
\langle \opt  (\sigma)f,g\rangle =
\langle
\sigma,W_{\tau }(g,f)\rangle, \quad f,g\in \mathcal{S}(\mathbb{R}^{d}).
\label{tauweak}
\end{equation}
The Weyl form $\mathrm{Op_{W}}(\sigma)$ of a pseudodifferential  operator can be recaptured when $\tau=1/2$, the Kohn-Nirenberg case $\mathrm{Op_{KN}}(\sigma)$ corresponds to   $\tau=0$.

Given $z=(x,\omega)\in\rdd$, we define the related  \tfs \,
acting on a function or distribution $f$ on $\rd$ as
\begin{equation}
\label{eq:kh25}
\pi (z)f(t) = e^{2\pi i \omega t} f(t-x), \, \quad t\in\rd.
\end{equation}
Let us recall the definition of a Gabor frame. Given a lattice  $\Lambda=A\zdd$,  with $A\in GL(2d,\R)$, and a non-zero window function $g\in L^2(\rd)$, we define the \emph{Gabor system}: $$\cG(g,\Lambda)=\{\pi(\lambda)g:\
\lambda\in\Lambda\}.$$
  The Gabor system $\cG(g,\Lambda)$   is called 
a Gabor frame, if there exist
constants $A,B>0$ such that
\begin{equation}\label{gaborframe}
A\|f\|_2^2\leq\sum_{\lambda\in\Lambda}|\langle f,\pi(\lambda)g\rangle|^2\leq B\|f\|^2_2,\qquad \forall f\in L^2(\rd).
\end{equation}
Fix $g\in \cS(\rd)\setminus \{0\}$. The \emph{Gabor matrix} of a linear continuous operator $T$ from $\cS(\rd)$ to $\cS'(\rd)$ is defined to be 
\begin{equation}\label{unobis2s} \langle T \pi(z)
g,\pi(u)g\rangle,\quad z,u\in \rdd.
\end{equation}
This Gabor matrix can be viewed as the kernel of an integral operator, cf. Section $2$ for details. 

For $\tau\in [0,1]$,  define the change of variables 
\begin{equation}\label{CL}
\mathcal{T}_{\tau}(z,u)=((1-\tau)z_1+\tau u_1,\tau z_2+(1-\tau)u_2),\quad z=(z_1,z_2), u=(u_1,u_2)\in\rdd.
\end{equation}

We possess all the instruments for the characterization of  $S^m(\rdd)$:
\begin{theorem}\label{Th3.3}
	Consider  $g\in \cS(\rd)\setminus\{0\}$ and a lattice $\Lambda$ such that  $\mathcal{G}\left(g,\Lambda\right)$ is a  Gabor frame for $L^{2}\left(\mathbb{R}^{d}\right)$. Fix $m\in\bR$.
	For any $\tau\in\left[0,1\right]$,  the following properties are equivalent:
	\begin{enumerate}
		\item[$(i)$] $\sigma\in S^m\left(\mathbb{R}^{2d}\right)$.
		\item[$(ii)$] $\sigma\in\mathcal{S}'\left(\mathbb{R}^{2d}\right)$ and for every $s\geq0$, $0<q\leq\infty$, there exists
		 a function $H_\tau\in L^q_{\la \cdot\ra^{s}}(\rdd)$, with 
		 \begin{equation}\label{H-tau}\|H_\tau\|_{L^q_{\la \cdot\ra^{s}}}\leq C, \quad  \forall\tau\in [0,1],\end{equation} such that  
		\begin{equation}\label{Ineq-Th.3.3-continuous}
		\left|\left\langle \opt \left(\sigma\right)\pi\left(z\right)g,\pi\left(u\right)g\right\rangle \right|\le H_\tau(u-z) \la\mathcal{T}_{\tau}(z,u)\ra^m, \qquad\forall u,z\in\mathbb{R}^{2d}.
		\end{equation}
		\item[$(iii)$] $\sigma\in\mathcal{S}'\left(\mathbb{R}^{2d}\right)$ and for every $s\geq0$  there exists
		a sequence $h_\tau\in \ell^q_{\la \cdot\ra^{s}}(\Lambda)$ with $\|h_\tau\|_{\ell^q_{\la \cdot\ra^{s}}}\leq C$ for every $\tau\in [0,1]$, such that  
		\begin{equation}\label{Ineq-Th.3.3-continuous2}
		\left|\left\langle \opt \left(\sigma\right)\pi\left(\mu\right)g,\pi\left(\lambda\right)g\right\rangle \right|\le h_\tau(\lambda-\mu)\la\mathcal{T}_{\tau}(\mu,\lambda)\ra^m, \qquad\forall\lambda,\mu\in\Lambda.
		\end{equation}
	\end{enumerate}
\end{theorem}

For  the  H\"{o}rmander class $S^0(\rdd)=S^0_{0,0}(\rdd)$, the Gabor matrix  characterization for Weyl operators was shown by Gr\"{o}chenig and Rzeszotnik in \cite[Theorem 6.2]{GR} (see also \cite{rochberg}) in the case $q=\infty$. So this result can be viewed as an extension to any $0<q\leq\infty$ and $\tau\in[0,1]$.

The central role in the proof of the result above is the  characterization of the class $S^m(\rdd)$ by an intersection of weighted modulation spaces (in particular, weighted Sj\"ostrand classes): for $0<q\leq\infty$,
$$
S^m(\rdd)=\bigcap_{s\geq 0} M^{\infty,q}_{\la \cdot\ra^{-m}\otimes \la \cdot\ra^s}(\rdd),
$$
cf.  Lemma \ref{charsm00}.

For the special case $m=0$, the  H\"{o}rmander class $S^0(\rdd)=S^0_{0,0}(\rdd)$ can also be represented as the intersection of Besov spaces and H\"{o}lder-Zygmund classes:
 $$S^0_{0,0}(\rdd)=\bigcap_{s\geq 0} \cC^s(\rdd)=  \bigcap_{s\geq 0} B^{\infty,q}_s(\rdd)= \bigcap_{s\geq 0} M^{\infty,q}_{1\otimes \la \cdot\ra^s}(\rdd),$$ 
 cf. Lemma \ref{Besov}, which   extends the characterization in \cite{GR}.

Observe that  $S^m$   contains the Shubin classes $\Gamma^m_\rho$, $0<\rho\leq1$, defined as \cite{shubin}
$$\Gamma^m_\rho(\rdd)=\{\sigma\in\cC^\infty(\rdd):\, |\partial^\alpha\sigma(z)|\leq C_\alpha \la z\ra^{m-\rho |\a|},\quad \a\in\mathbb{N}^{2d},\,\, z\in\rdd\}, $$ 
and can be viewed as their limit case $\rho=0$.
 The Shubin classes enjoy a symbolic calculus  very useful when dealing with the corresponding pseudodifferential operators. This is not the case of  $S^m(\rdd)$.  Hence, the characterization in Theorem \ref{Th3.3} might be an instrument to infer boundedness, composition, inversion properties of the corresponding operators in suitable function spaces, such as the modulation ones.  
 

As a byproduct, Theorem  \ref{Th3.3}   allows to compute the Gabor matrix decay of a Born-Jordan operator. We present  some continuity properties of the latter on weighted modulation spaces, extending the work \cite{CdNBJ-JFA}.

This study paves the way to other possible investigations. For instance, when the symbol $\sigma$ on $\rdd$
satisfies a Geverey-type regularity of order $s>0$:
\begin{equation}\label{C}
|\partial^{\a} \sigma(z)|\lesssim M(z) C^{|\alpha|}(\alpha!)^s,\quad\ \a\in\mathbb{N}^{2d},\
z\in\rdd,
\end{equation}
with $M$ any possible $v$-moderate weight (see Section $2$ for its definition).  These symbols were applied in \cite{CNRExsparse2015} to investigate the sparsity of the Gabor-matrix representation of Fourier integral operators.
In this case we conjecture that the right modulation spaces to be considered are of the type $M^{\infty,q}_{M\otimes e^{-\eps|\cdot|^{1/s}}}(\rdd)$.

Eventually, one might extend the characterization exhibited in Theorem \ref{Th3.3} to Fourier integral operators of Schr\"odinger-type with symbols in $S^m$ and suitable phases as in \cite{CGRNFIOJMPA2013}. This will be the object of a further work.

The paper is organized as follows. In Section $2$ we present the function spaces object of our study. In particular, we focus on modulation spaces and present the properties needed for our results. We then prove the characterization of the classes $S^m(\rdd)$ and in particular of the {H}\"{o}rmander classes $S^0_{0,0}(\rdd)$. Section $3$ is devoted to the study of the Gabor matrix for $\tau$-operators and Born-Jordan operators. As an application, boundedness results on modulation spaces are exhibited.

\section{Function spaces and preliminaries} 
 In this manuscript $\hookrightarrow$ denotes the continuous embeddings of function spaces. Recall that the conjugate exponent $p^{\prime }$ of $p\in \lbrack
1,\infty ]$ is defined by $1/p+1/p^{\prime }=1$.

The notation $y\omega$  means the inner product  $y\cdot\omega$, $|x|$ stands for the Euclidean norm of $x$ and $x^2$  means $|x|^2$.

We denote by $v$  a
continuous, positive,  submultiplicative  weight function on $\rd$, i.e., 
$ v(z_1+z_2)\leq v(z_1)v(z_2)$, for all $ z_1,z_2\in\Ren$.
We say that $w\in \mathcal{M}_v(\rd)$ if $w$ is a positive, continuous  weight function  on $\Ren$  {\it
	$v$-moderate}:
$ w(z_1+z_2)\leq Cv(z_1)w(z_2)$  for all $z_1,z_2\in\Ren$ (or for all $z_1,z_2\in \zd$).
We will mainly work with polynomial weights of the type
\begin{equation}\label{vs}
v_s(z)=\la z\ra^s =(1+|z|^2)^{s/2},\quad s\in\bR,\quad z\in\rd\,\, (\mbox{or}\, \zd).
\end{equation}
Moreover, we limit to weights $w$ with at most polynomial growth, that is there exist $C>0$ and $s>0$ such that
\begin{equation}\label{growth}
w(z)\leq C\la z\ra^s,\quad z\in\rd.
\end{equation}\par 
We shall work mostly with weights on $\rdd$ or $\zdd$; we define $(w_1\otimes w_2)\phas\coloneqq w_1(x)w_2(\o)$, for $w_1,w_2$ weights on $\rd$.\par

{\bf Spaces of sequences.} 	For $0<p\leq \infty$, $w\in \mathcal{M}_v(\zd)$, the space $\ell^{p}_w(\zd)$ consists of all sequences $a=(a_{k})_{k\in\zd}$ for which the (quasi-)norm 
	$$\|a\|_{\ell^{p}_w}=\left(\sum_{k\in\zd}|a_{k}|^p w(k)^p\right)^{\frac 1p}
	$$
	(with obvious modification for $p=\infty$) is finite.

We are going to use the following inclusion relations for $w(k)=\la k\ra^s$, $s\geq0$:
	 If $0<p_1,p_2 \leq\infty$, with 
	 $$s_2\leq s_1,\quad \frac1{p_2}+\frac{s_2}{d}<\frac1{p_1}+\frac{s_1}{d},$$
	 then \begin{equation}\label{inclusion}
	 \ell^{p_2}_{\la k\ra^{s_2}}(\zd)\hookrightarrow\ell^{p_1}_{\la k\ra^{s_1}}(\zd).
	 \end{equation}

The so-called translation and
modulation operators are defined by $T_x
g(y)=g(y-x)$ and $M_\o g(y)=e^{2\pi
	i\o y}g(y)$, respectively. Let $\cS(\rd)$ be the Schwartz class and consider $g\in\cS(\rd)$  a
non-zero window function. The
the  short-time Fourier
transform (STFT) $V_gf$ of a
function/tempered distribution $f$ in $\cS'(\rd)$ with
respect to the the window $g$ is defined by
\[
V_g f(x,\o)=\la f, M_{\o}T_xg\ra =\int e^{-2\pi i \o y}f(y)\overline{g(y-x)}\,dy,
\]
(i.e.,  the  Fourier transform $\cF$
applied to $f\overline{T_xg}$). \par


{\bf Modulation Spaces.} For $1\leq p,q\leq \infty$ such spaces were introduced  by H. Feichtinger in \cite{F1}, then extended to $0<p,q\leq\infty$ by Y.V. Galperin and S. Samarah in \cite{Galperin2004}. Their main properties and applications are now available in several textbooks, see for instance \cite{Elena-book}.
\begin{definition}\label{def2.4}
	Fix a non-zero window $g\in\cS(\rd)$, a weight $w\in\mathcal{M}_v(\rdd)$ and $0<p,q\leq \infty$. The modulation space $M^{p,q}_w(\rd)$ consists of all tempered distributions $f\in\cS'(\rd)$ such that the (quasi-)norm 
	\begin{equation}\label{norm-mod}
	\|f\|_{M^{p,q}_w}=\|V_gf\|_{L^{p,q}_w}=\left(\intrd\left(\intrd |V_g f \phas|^p w\phas^p dx  \right)^{\frac qp}d\o\right)^\frac1q 
	\end{equation}
	(obvious changes with $p=\infty$ or $q=\infty)$ is finite. 
\end{definition}
They  are quasi-Banach spaces (Banach spaces whenever $1\leq p,q\leq\infty$), whose (quasi-)norm does not depend on the window $g$, in the sense that different non-zero window functions in $\cS(\rd)$ yield equivalent (quasi-)norms. Moreover, if $1\leq p,q\leq\infty$, the window class $\cS(\rd)$ can be extended  to the modulation space  $M^{1,1}_v(\rd)$ (so-called Feichtinger algebra). 

To be short, we write $M^p_w(\rd)$ in place of $M^{p,p}_w(\rd)$ and $M^{p,q}(\rd)$ if $w\equiv 1$. 

We  recall the inversion formula for
the STFT: assume $g\in M^{1}_v(\rd)\setminus\{0\}$,
$f\in M^{p,q}_w(\rd)$, with $w\in\mathcal{M}_v(\rdd)$, then
\begin{equation}\label{invformula}
f=\frac1{\|g\|_2^2}\int_{\R^{2d}} V_g f(z) \pi (z)  g\, dz \, ,
\end{equation}
and the  equality holds in $M^{p,q}_w(\rd)$. The adjoint operator of $V_g$,  defined by
$$V_g^\ast F(t)=\intrdd F(z)  \pi (z) g dz \, ,
$$
maps the mixed-norm space $L^{p,q}_w(\rdd)$ into $M^{p,q}_w(\rd)$. In particular, if $F=V_g f$ the inversion formula \eqref{invformula} can be rephrased as
\begin{equation}\label{treduetre}
{\rm Id}_{M^{p,q}_w}=\frac 1 {\|g\|_2^2} V_g^\ast V_g.
\end{equation}

We need to introduce an alternative definition of modulation spaces  we shall use  in the sequel.
For $k\in \zd$, we denote by $\Q_k$ the unit closed cube centred at $k$.  The family $\{\Q_k\}_{k\in\zd}$ is a covering of $\rd$. We define $|\xi|_{\infty}:=\max_{i=1,\dots,d} |\xi_i|$, for $\xi\in\rd$. Consider now a smooth function $\rho:\rd \to [0,1]$ satisfying $\rho(\xi)=1$ for $|\xi|_{\infty}\leq 1/2$ and $\rho(\xi)=0$ for $|\xi|_{\infty}\geq 3/4$. Define 
\begin{equation}\label{C2rhotransl}
\rho_k(\xi)= T_k\rho (\xi)=\rho(\xi-k),\quad k\in\zd,
\end{equation}
that is,  $\rho_k$ is the translation of $\rho$ at $k$. By the assumption on $\rho$, we infer that $\rho_k(\xi)=1$ for $\xi \in \Q_k$ and 
$$\sum _{k\in\zd} \rho_k(\xi) \geq 1,\quad \forall\,\xi\in\rd. 
$$
Denote by 
\begin{equation}\label{C2sigmak}
\sigma_k(\xi)=\frac{\rho_k(\xi)}{\sum _{l\in\zd} \rho_l(\xi)},\quad \xi\in\rd,\, k\in\zd.
\end{equation}
Observe that $\sigma_k(\xi)=\sigma_0(\xi-k)\in \cD(\rd)$ and the sequence $\{\sigma_k\}_{k\in\zd}$ is a smooth partition of unity
$$\sum_{k\in\zd} \sigma_k(\xi)=1,\quad \forall \xi\in\rd.
$$ 
 For $k\in\zd$, we define the frequency-uniform decomposition operator by
	\begin{equation}\label{C2freq-unifdecop}
	\Box_k\coloneqq \cF^{-1} \sigma_k \cF.
	\end{equation}
The previous operators allow to introduce an alternative (quasi-)norm on the weighted modulation spaces $M_{h\otimes w}^{p,q}(\rd)$ inspired by \cite{baoxiang} as follows.
\begin{proposition}\label{C2altnorm}
 For $0< p,q\leq \infty$, $h,w\in\mathcal{M}_v(\rd)$ have	\begin{equation}\label{C2tildeMpq}
	\|f\|_{{M}^{p,q}_{h\otimes w}(\rd)}\asymp\left(\sum_{k\in\zd} \|\Box_k f\|^q_{L^p_h}w(k)^q\right)^\frac 1q,\quad f\in\cS'(\rd),
	\end{equation}
	with obvious modification for $q=\infty$.
\end{proposition}
\begin{proof}
	The case $p,q\geq 1$ is well known, see for example \cite[Proposition 2.3.25]{Elena-book}. The cases $0<p<1$ or $0<q<1$ are an easy modification of that proof. Namely, let us point out the main changes. If $0<p\leq1$, we consider
	$$\Box_k f=\cF^{-1}\sigma_k \cF f = \cF^{-1}\sigma_k T_\xi \bar{\hat{\phi}}\cF f,\quad \mbox{for}\quad \xi \in \Q_k,
	$$
	since $T_\xi \bar{\hat{\phi}}=1$ in supp $\sigma_k$ for $\xi \in \Q_k$.
	Using Young's inequality for distributions compactly supported in the frequencies (see \cite[Lemma 2.6]{Kobayashi2006}, which holds also for $L^p_h$, $0<p\leq 1$, with $h$ being $v$-moderate), for $\xi\in \Q_k$,  we obtain
	$$\|\Box_k f\|_{L^p_h}\lesssim \|\cF^{-1}\sigma_k\|_{L^p_v}\|\cF^{-1}T_\xi \bar{\hat{\phi}}\cF f\|_{L^p_h}\lesssim \|\cF^{-1}T_\xi \bar{\hat{\phi}}\cF f\|_{L^p_h}.
	$$
	The rest of the proof is analogous to the Banach case and we leave the details to the interested reader.
\end{proof}

An useful embedding is contained in what  follows. 
\begin{proposition}\label{C2inclweightmod} Given $0< p_1,p_2,q_1,q_2\leq\infty$,  with $m,s_1,s_2$ in $\bR$,   one has
	\begin{equation}\label{C2inclmodweight}
	M^{p_1,q_1}_{\la\cdot\ra^m\otimes \la\cdot\ra^{s_1}}(\rd)\hookrightarrow 	M^{p_2,q_2}_{\la\cdot\ra^m\otimes\la\cdot\ra^{s_2}}(\rd)
	\end{equation}
	if and only if
	\begin{equation}\label{C2inclmodweightindicicesp}
	p_1\leq p_2
	\end{equation}
	and 
	\begin{equation}\label{C2inclmodweightindicicesq}
	q_1\leq q_2, \,\,s_1\geq s_2 \quad \quad\mbox{or}\quad \quad 
	q_1>q_2,\quad  \frac{s_1}{d}+\frac{1}{q_1}> \frac{s_2}{d}+\frac{1}{q_2}.
	\end{equation}
\end{proposition}
\begin{proof}
	The Banach case when $m=0$ was originally shown by H. Feichtinger in \cite{F1}. We use similar arguments as in that proof.  The discrete modulation norm defined in \eqref{C2tildeMpq} is given by
	$$\|f\|_{\mpq_{\la\cdot\ra^m\otimes \la\cdot\ra^s}}\asymp \left(\sum_{k\in\zd}\|\square_k f\|_{L^p_{\la\cdot\ra^m}}^q\la k\ra ^{sq}\right)^\frac1q.
	$$
	The necessity of  \eqref{C2inclmodweightindicicesp} follows from the fact that $\cF L^{p_1}$ is locally contained in $\cF L^{p_2}$ if and only if $p_1\leq p_2$ (with strict inclusion if $p_1< p_2$), cf. \cite{Bloom1982,Fournier1973,Kobayashi2006,triebel2010theory}. The  set of conditions in \eqref{C2inclmodweightindicicesq} in turn describes the inclusions between weighted
	$\ell^q$ spaces: $\ell^{q_1}_{\la\cdot\ra^{s_1}}\subset \ell^{q_2}_{\la\cdot\ra^{s_2}}$
	if and only if the indices' relations in \eqref{C2inclmodweightindicicesq} are satisfied, cf. for instance \cite[Lemma 2.10]{Guo-incl-mod-Triebel2017}. This concludes the proof.
\end{proof}



We also recall the following inclusion relations, see e.g. \cite[Theorem 2.4.17]{Elena-book} or \cite[Theorem 3.4]{Galperin2004}: If $p_{1}\leq p_{2}$, $q_{1}\le q_{2}$ and $w_{2}\lesssim w_{1}$,
then 
\begin{equation}\label{inclmod}
M_{w_{1}}^{p_{1},q_{1}}(\rd)\hookrightarrow M_{w_{2}}^{p_{2},q_{2}}(\rd). 
\end{equation}
\begin{corollary}\label{inclquasi-Banachq}
	For $0<q_1\leq q_2\leq\infty$, $d\in\bN_+$, $m,s,r\in\bR$,  $r>s+d(1/q_1-1/q_2)$, we have the following continuous embeddings:\begin{equation}\label{modonquasiBan}
	M^{\infty,q_1}_{\la\cdot\ra^m\otimes \la\cdot\ra^r}(\rd)\hookrightarrow M^{\infty,q_2}_{\la\cdot\ra^m\otimes \la\cdot\ra^r}(\rd)\hookrightarrow M^{\infty,q_1}_{\la\cdot\ra^m\otimes \la\cdot\ra^s}(\rd).
	\end{equation}
\end{corollary}
\begin{proof}
The first embedding  is a straightforward application of the inclusion relations in \eqref{inclmod}.  The second one follows by the embedding in Proposition \ref{C2inclweightmod}.
\end{proof}

{\bf Besov Spaces.}  The Besov spaces are denoted by $B^{p,q}_s(\rd)$, $0<p,q\leq\infty$, $s\in\bR$, and defined as follows. Suppose that $\psi_0,\psi\in\cS(\rd)$ satisfy $\supp \,\psi_0\subset \{\o\in\rd:\,|\o|\leq 2\}$,  $\supp \,\psi\subset \{\o\in\rd:\,1/2\leq |\o|\leq 2\}$ and  $\psi_0(\o)+\sum_{j=1}^\infty \psi(2^{-j}\o)=1$ for every $\o\in\rd$. Set $\psi_j(\o):=\psi(2^{-j}\o)$, $\o\in\rd$. Then the Besov space $B^{p,q}_s(\rd)$ consists of all tempered distributions $f\in\cS'(\rd)$ such that the (quasi-)norm
\begin{equation}\label{norm-besov}
\|f\|_{B^{p,q}_s}=\left(\sum_{j=0}^\infty 2^{jsq} \|\cF^{-1}(\psi_j \cF f)\|_{p}^q\right)^{1/q}<\infty
\end{equation}
(with usual modifications when $q=\infty$).
Besov spaces are generalizations of both H\"{o}lder-Zygmund and Sobolev spaces, see e.g. \cite{triebel2010theory}. Precisely, we recapture the Sobolev spaces when $p=q=2$, $s\in\bR$: $B^{2,2}_s(\rd)=H^s(\rd)$. For $s>0$, $B^{\infty,\infty}_s(\rd)=\cC^s(\rd)$, the H\"{o}lder-Zygmund classes, whose definition is as follows. For $s>0$, we can write $s=n+\eps$, with $n\in\bN$ and $\eps<1$. Then $\cC^s(\rd)$ is the space of functions $f\in\cC^n(\rd)$ such that  for each multi-index $\a\in\bN^d$, with $|\a|=n$, the derivative $\partial^\a f$ satisfies the  H\"{o}lder condition
$|\partial^\a f(x)-\partial^\a f(y)|\leq K |x-y|^\eps$, for a suitable $K>0$.

Inclusion relations between modulation and Besov spaces $B^{\infty,q}_{s}$ were first obtained for $1\leq q\leq\infty$ (the Banach setting) in \cite[Theorem 2.10]{Toftweight2004} and then for $0<q\leq\infty$ in \cite{baoxiang}:   for $0<q\leq\infty$, set $\theta(q)=\min\{0, 1/q-1\}$, then 
\begin{equation}\label{besovmodone}
B^{\infty,q}_{s+d/q}(\rd)\hookrightarrow M^{\infty,q}_{1\otimes \la\cdot\ra^s}(\rd)\hookrightarrow B^{\infty,q}_{s+d\theta(q)}(\rd), \quad s\in\bR.
\end{equation}

\subsection{Gabor analysis of $\tau$-pseudodifferential operators} 	
For any fixed $m\in\bR$, the class $S^m(\rdd)$ in \eqref{Sm00} is a Fr\'echet space when endowed with the sequence of norms $\{|\cdot|_{N,m}\}_{N\in\bN}$,
\begin{equation}\label{semihormander}
|\sigma|_{N,m}:=\sup_{|\a|\leq N} \sup_{z\in\rdd}|\partial^\alpha\sigma(z)|\la z\ra^{-m},\quad N\in\bN. 
\end{equation}
For $n\in\bN$, $m\in\bR\setminus\{0\}$, we define by $\cC_m^n(\rdd)$ the space of functions having $n$ derivatives and satisfying \eqref{semihormander} for $N=n$, whereas $\cC^n(\rdd)$ is the space of functions with $n$ bounded derivatives.
Clearly we have the equalities
$$ S^m(\rdd)=\bigcap_{n\geq 0} \cC_m^n(\rdd), \,m\in\bR\setminus\{0\},\quad S^0(\rdd)=\bigcap_{n\geq 0} \cC^n(\rdd).
$$

A characterization of the class $S^0(\rdd)=S^0_{0,0}(\rdd)$ with modulation spaces was announced by Toft in \cite[Remark 3.1]{toft2007} and proved in \cite[Lemma 6.1]{GR}.
\begin{lemma}\label{lemma2.1} We have the equalities
	\begin{equation}
		\bigcap_{n\geq 0} \cC^n(\rd)=\bigcap_{s\geq 0} M^\infty_{1\otimes \la \cdot\ra^s}(\rd)=\bigcap_{s\geq 0} M^{\infty,1}_{1\otimes \la \cdot\ra^s}(\rd). 
	\end{equation}
	Hence $S^0(\rdd)=\bigcap_{s\geq 0} M^\infty_{1\otimes \la \cdot\ra^s}(\rdd)=\bigcap_{s\geq 0} M^{\infty,1}_{1\otimes \la \cdot\ra^s}(\rdd)$. 
\end{lemma}
In what follows we extend the previous outcome to all the classes $S^m(\rdd)$, $m\in\bR$. 
\begin{lemma}\label{charsm00} For $m\in \bR$,  $0<q\leq\infty$, $n\in\bN$, $s\in(0,+\infty)$, we have the equalities of  Fr\'echet spaces
	\begin{equation}\label{e2}
	S^m(\rdd)=	\bigcap_{n\geq 0} \cC^n_m(\rdd)=\bigcap_{n\geq 0} M^{\infty,q}_{\la \cdot\ra^{-m}\otimes \la \cdot\ra^n}(\rdd)=\bigcap_{s\geq 0} M^{\infty,q}_{\la \cdot\ra^{-m}\otimes \la \cdot\ra^s}(\rd)
	\end{equation} with  equivalent families of  (quasi-)norms
\begin{equation}\label{seqnorm}
	\{|\cdot|_{n,m}\}_{n\in\bN},\quad \{\|\cdot\|_{M^{\infty,q}_{\la \cdot\ra^{-m}\otimes\la\cdot\ra^n}}\}_{n\in\bN}, \quad \{\|\cdot\|_{M^{\infty,q}_{\la \cdot\ra^{-m}\otimes\la\cdot\ra^s}}\}_{s\geq 0}.
	\end{equation}  
	In particular, for every $n\in\bN$,
	\begin{equation}\label{e1}
	\|f\|_{M^\infty_{\la\cdot\ra^{-m}\otimes\la \cdot\ra^n}}\leq C(n,m) |f|_{n,m}. 
	\end{equation}
\end{lemma}

\begin{proof}  The equality  	$S^m(\rdd)=\bigcap_{n\geq 0} M^{\infty,1}_{\la \cdot\ra^{-m}\otimes \la \cdot\ra^n}(\rdd)$ was proved in \cite[Remark 2.18]{HolstToftWahlberg2007}. The embeddings in \eqref{modonquasiBan} then give the equalities in \eqref{e2} with the equivalent families of (quasi-)norms in \eqref{seqnorm}.\par
		Let us show the estimate \eqref{e1}.
	For $f\in \cC_m^n(\rd)$ ($\cC^n(\rd)$ if $m=0$) and any multi-index $\alpha\in\bN^d$ with $|\a|\leq n$, we consider the function $\partial^\a(fT_x\bar{g})$. Taking its Fourier transform we get
	\begin{equation}\label{evgf}
	\cF(\partial^\a(fT_x\bar{g}))(\o)=(2\pi i \o)^\a\cF(fT_x\bar{g})(\o)=(2\pi i \o)^\a V_gf\phas.
	\end{equation}
	In what follows we use the boundedness of $\cF: L^{1}(\rd)\to \cC_0(\rd)$,  Peetre's inequality $\la x\ra^{-m}\leq 2^{-m} \la x-t\ra^{|m|}\la t\ra^{-m}$, and Leibniz' formula:
	\begin{align*}
	\la x\ra^{-m} \|\cF(\partial^\a(fT_x\bar{g}))\|_{\infty}&\leq \la x\ra^{-m} \|\partial^\a(fT_x\bar{g})\|_1\\
	&= \left\|\la x\ra^{-m}\sum_{\beta\leq  \alpha}\binom \alpha \beta \partial^\beta f \,T_x\partial^ {\alpha-\beta}\bar{g}\right\|_1\\
	&\leq 2^{-m}\sum_{\beta\leq  \alpha}\binom \alpha \beta \|(\partial^\beta f)\la \cdot\ra^{-m}\|_\infty\|(\partial^ {\alpha-\beta}\bar{g})\la\cdot\ra^{|m|}\|_1\\
	&\leq 2^{-m}\sup_{|\beta|\leq n}\|(\partial^\beta f)\la \cdot\ra^{-m}\|_\infty M_\al\max_{\beta\leq  \alpha}\binom \alpha \beta \|(\partial^ {\alpha-\beta}\bar{g})\la\cdot\ra^{|m|}\|_1\\
	&= C_{\a,g,m}|f|_{n,m},
	\end{align*}
	where $ C_{g,\a,m}=2^{-m}M_\al\max_{\beta\leq  \alpha}\binom \alpha \beta \|(\partial^ {\alpha-\beta}\bar{g})\la\cdot\ra^{|m|}\|_1$ with $M_\al=\#\{\be\in\bN^d,\be\leq\al\}$. The estimate above and formula \eqref{evgf} yield
	\begin{equation}\label{evgf2}
	\sup_{x\in\rd} |V_g f\phas| \la x\ra^{-m}\leq C_{g,\a,m}|f|_{n,m}|\o^\a|^{-1},\quad |\o|\not=0,\quad\forall |\a|\leq n.
	\end{equation}
	Now if $f\in \bigcap_{n\geq 0}\cC_m^n(\rd)$ then for every $\a\in\bN^d$ there exists $C=C_\alpha>0$ such that the estimate in \eqref{evgf2} holds true.  Since $\la \o\ra^n\leq \sum_{|\a|\leq n}c_\a|\o^\a|$ for suitable $c_\a\geq 0$, we obtain
	$$\sup_{x,\o\in\rd}|V_g f\phas|\la x\ra^{-m}\la \o\ra^n\leq C|f|_{n,m},\quad\forall n\geq0$$
	for a suitable $C=C(n,m)>0$ that is \eqref{e1}.
	\end{proof}

	In particular, for $m=0$ we recapture the outcome of Lemma  \ref{lemma2.1}.

For the case $m=0$ we can characterize  the H\"{o}rmander class $S^0(\rdd)=S_{0,0}^0(\rdd)$ by H\"{o}lder-Zygmund classes $\cC^s(\rdd)=B^{\infty,\infty}_s(\rdd)$ and by  Besov spaces. 

\begin{lemma}\label{Besov}
For $0<q\leq\infty$, we have the equalities 
\begin{equation}\label{eqfrechetsobolev}
S^0_{0,0}(\rdd)=\bigcap_{s\geq 0} \cC^s(\rdd)=  \bigcap_{s\geq 0} B^{\infty,q}_s(\rdd)= \bigcap_{s\geq 0} M^{\infty,q}_{1\otimes \la \cdot\ra^s}(\rdd),
\end{equation}
with equivalent families of (quasi-)norms 
\begin{equation}\label{seqnormsobolev}
\{\|\cdot\|_{B^{\infty,\infty}_s}\}_{s\geq 0},\quad\{\|\cdot\|_{B^{\infty,q}_s}\}_{s\geq 0},\quad \{\|\cdot\|_{M^{\infty,q}_{1\otimes\la\cdot\ra^s}}\}_{s\geq 0}.
\end{equation}
\end{lemma}
\begin{proof}
It is a straightforward consequence of Lemma \ref{charsm00} and the inclusion relations in  \eqref{besovmodone}. 
\end{proof}


\section{Gabor matrix decay}\label{pseudo}
Let us first represent the Gabor matrix as a kernel of an integral operator. Consider a linear and bounded operator $T$ from $\cS(\rd)$ into $\cS'(\rd)$. 
The inversion formula \eqref{treduetre} for $g\in M^1_v(\rd)$, $\|g\|_2=1$ is simply $V_g^\ast V_g={\rm Id}$. The operator   $T$ can be written as
\begin{equation}\label{e8}
T=V_g^\ast V_g T V_g^\ast V_g.
\end{equation}
The linear transformation $V_g T V_g^\ast$ is an integral operator with kernel $K_T$ given by the Gabor matrix of $T$:
$$K_T(u,z)=\la T\pi(z)g,\pi(u)g\ra,\quad u,z\in\rdd \, .
$$
By definition  and the inversion formula, $V_g$ is bounded from $M^{p,q}_w(\rd )$ to
$L^{p,q}_w(\rdd )$ and $V_g^*$  from $L^{p,q}_w(\rdd )$ to
$M^{p,q}_w(\rd )$. Hence the continuity properties  of $T$ on modulation spaces can be obtained by the corresponding ones of the operator $V_g T V_g^\ast$ on mixed-norm $L^{p,q}_w$ spaces. These issues will be studied in  Proposition \ref{Proposition-continuity-Shubin} and Corollary \ref{Cor-Continuity-BJ} and  can be  achieved by studying the Gabor matrix decay of $T$.

First, we focus on  the characterization of the Gabor matrix of $\opt(\sigma)$.

\begin{proposition}\label{Pro-uniform-bounds-via-STFTtauWDgaussian}
	Consider $0< p,q\leq\infty$, $\tau\in[0,1]$, $w\in\mathcal{M}_{v}(\bR^{4d})$ satisfying \eqref{growth}, $G\in\cS(\rdd)\setminus\{0\}$, $g\in\cS(\rd)\setminus\{0\}$ and define $ \Phi_\tau\coloneqq W_\tau(g,g)$.
	Then there exist $A= A(v,g,G)>0$, $B= B(v,g,G)>0$ such that
	\begin{equation}
	A \norm{V_G \sigma}_{L^{p,q}_{w}}\leq\norm{V_{\Phi_\tau} \sigma}_{L^{p,q}_{w}}\leq B\norm{V_G \sigma}_{L^{p,q}_{w}},
	\end{equation}
	for every $\tau\in[0,1]$ and $\sigma\in M^{p,q}_{w}(\rdd)$.
\end{proposition}
\begin{proof}
	By Proposition 2.2 and Remark 2.3 in \cite{deGossonToft-BJ2020}
	the mapping
	$$(\tau, f,g)\mapsto W_\tau(f,g)$$ 
		is continuous from $\bR\times\cS(\rd)\times\cS(\rd)$ to $\cS(\rdd)$
		 and locally uniformly bounded. Since $\Phi_\tau$ for $\tau\in [0, 1]$ belongs to a bounded set
	in $\cS(\rdd)$, the result follows  immediately from \cite[Theorem
	11.3.7]{book} for $p,q\geq 1$ and \cite[Theorem 3.1]{Galperin2004} for $0<p,q\leq\infty$. 
\end{proof}

Finally, we need the following result  for $\tau$-pseudodifferential operators  \cite[Lemma 4.1]{CNT}.
\begin{lemma}\label{lem:STFT-gaborm}
	Fix a window $g\in \cS(\rd)\setminus\{0\}$
	and define $\Phi_{\tau}=W_{\tau}(g,g)$ for $\tau\in\left[0,1\right]$.
	Then, for $\sigma\in \cS'\left(\mathbb{R}^{2d}\right)$,
	\begin{equation}
	\left|\left\langle \opt \left(\sigma\right)\pi\left(z\right)g,\pi\left(u\right)g\right\rangle \right|=\left|{V}_{\Phi_{\tau}}\sigma\left(\mathcal{T}_{\tau}\left(z,u\right),J\left(u-z\right)\right)\right|\label{eq:gaborm as STFT}.
	\end{equation}
	where $z=(z_1,z_2)$, $u=(u_1,u_2)$,  the operator $\mathcal{T}_{\tau}$ is defined in \eqref{CL} and $J$ is given by
	\begin{equation*}
	J(z)=(z_2,-z_1).
	\end{equation*}
\end{lemma}

We are ready to state the characterization of $\tau$-operators with symbols in  $M_{\la \cdot\ra^{-m}\otimes\la \cdot\ra^s}^{\infty,q}(\mathbb{R}^{2d})$.
\begin{theorem}\label{teor41}
	Consider  $g\in \cS(\rd)\setminus\{0\}$ and a lattice $\Lambda \subset\mathbb{R}^{2d}$ such that  $\mathcal{G}\left(g,\Lambda\right)$ is a  Gabor frame for $L^{2}\left(\mathbb{R}^{d}\right)$.
	For  $\tau\in [0,1]$, let $\mathcal{T}_{\tau}$ be the linear transformation defined in \eqref{CL}. For any $s,m\in\bR$, $0<q\leq\infty$, the following properties are equivalent:
	\begin{enumerate}
		\item[$(i)$] $\sigma\in M_{\la \cdot\ra^{-m}\otimes\la \cdot\ra^s}^{\infty,q}\left(\mathbb{R}^{2d}\right)$.
		\item[$(ii)$] $\sigma\in\mathcal{S}'\left(\mathbb{R}^{2d}\right)$ and there exists
		a function $H_\tau\in L^q_{\la \cdot\ra^{s}}(\rdd)$ satisfying \eqref{H-tau} such that
		\begin{equation}
		\left|\left\langle \opt \left(\sigma\right)\pi\left(z\right)g,\pi\left(u\right)g\right\rangle \right|\le  H_\tau(u-z)\la\mathcal{T}_{\tau}(z,u)\ra^m, \qquad\forall u,z\in\mathbb{R}^{2d}.\label{eq:almdiag J}
		\end{equation}
		\item[$(iii)$] $\sigma\in\mathcal{S}'\left(\mathbb{R}^{2d}\right)$ and there exists
		 a sequence $h_\tau\in \ell^q_{\la \cdot\ra^{s}}(\Lambda)$ with $\|h_\tau\|_{\ell^q_{\la \cdot\ra^{s}}}\leq C$, for every $\tau\in [0,1]$ such that
		\begin{equation}
		\left|\left\langle \opt \left(\sigma\right)\pi\left(\mu\right)g,\pi\left(\lambda\right)g\right\rangle \right|\le C {h_\tau( \lambda-\mu) \la\mathcal{T}_{\tau}(\mu,\lambda)\ra^m}, \qquad\forall\lambda,\mu\in\Lambda.\label{eq:almdiag discr}
		\end{equation}
	\end{enumerate}
\end{theorem}
\begin{proof}
	The proof follows the pattern of the corresponding one for Weyl operators with symbols in weighted Sj\"ostrand's classes \cite[Theorem 3.2]{grochenig2}.\\
	$(i)\Rightarrow (ii)$ This implication comes easily from the characterization \eqref{eq:gaborm as STFT}.
	In details, observing that $\la Ju\ra=\la u\ra$,
	\begin{align*}
	\left|\left\langle \opt \left(\sigma\right)\pi\left(z\right)g,\pi\left(u\right)g\right\rangle \right|&=\left|{V}_{\Phi_{\tau}}\sigma\left(\mathcal{T}_{\tau}\left(z,u\right),J\left(u-z\right)\right)\right|\\
	&\leq\sup_{w\in\rdd}\left(|V_{\Phi_\tau}\sigma|(w,J\left(u-z\right)| \la w\ra^{-m}\right)\la\mathcal{T}_{\tau}(z,u)\ra^{m}\\
	&=H_\tau(u-z)\la\mathcal{T}_{\tau}(z,u)\ra^{m},
	\end{align*}
	where
	$$H_\tau(u):= \sup_{w\in\rdd}\left(|V_{\Phi_\tau}\sigma|(w,Ju)|\la w\ra^{-m}\right).$$
	For $0<q<\infty$,
	$$\|H_\tau\|_{L^q_{{\la\cdot\ra}^{s}}}=\left(\intrdd \left[\sup_{w\in\rdd}\left(|V_{\Phi_\tau}\sigma|(w,Ju)|\la w\ra^{-m}\right)\right]^q\la u\ra^{qs} du\right)^{\frac1q}\asymp\|\sigma\|_{M_{\la \cdot\ra^{-m}\otimes\la \cdot\ra^s}^{\infty,q}},$$
	Hence by Proposition \ref{Pro-uniform-bounds-via-STFTtauWDgaussian} we obtain the estimate \eqref{H-tau}.
	The case $q=\infty$ is analogous. \par\noindent 
	$(ii)\Rightarrow (i)$ Consider the change of variables $y=\mathcal{T}_{\tau}(z,u)$ and $t=J(u-z)$, so that
	\begin{equation}\label{Eq-change-variables}
		\begin{cases}
		z(y,t)&=y-U_\tau J^{-1}t\\
		u(y,t)&=y+(I_{2d}-U_\tau)J^{-1}t
		\end{cases},\qquad
		U_\tau z\coloneqq 
		\begin{bmatrix}
		\tau I_d & 0\\
		0 & (1-\tau)I_d
		\end{bmatrix}
		z=\mathcal{T}_\tau(0,z)
	\end{equation}
	and $u(y,t)-z(y,t)=J^{-1}t$.	For $0<q<\infty$, using \eqref{eq:gaborm as STFT} and \eqref{eq:almdiag J},
	\begin{align*}
	\norm{\sigma}_{M_{\la \cdot\ra^{-m}\otimes\la \cdot\ra^s}^{\infty,q}}&\asymp	\left(\intrdd\left(\sup_{y\in\rdd}\left|{V}_{\Phi_{\tau}}\sigma\left(y,t\right)\right|\la y\ra^{-m}\right)^q\la t\ra^{qs}dt\right)^{\frac1q}\\
	&= \left(\intrdd\left(\sup_{y\in\rdd}\left|\left\langle \opt \left(\sigma\right)\pi\left(z(y,t)\right)g,\pi\left(u(y,t)\right)g\right\rangle \right|\la \mathcal{T}_{\tau}(z,u)\ra^{-m}\right)^q\la t\ra^{qs}dt\right)^{\frac1q}\\
	&\leq \left(\intrdd\left|
	H_\tau(J^{-1}t) \right|^q\la t\ra^{qs}dt\right)^{\frac1q}\\
	&\leq C,
	\end{align*}
	where we used \eqref{H-tau}. The case $q=\infty$ is analogous. \\
	$(ii)\Leftrightarrow(iii)$ The argument requires that $\mathcal{G}\left(g,\Lambda\right)$ is a  Gabor frame for $L^{2}\left(\mathbb{R}^{d}\right)$.
	Then the equivalence can be proved similarly  to \cite[Theorem 3.1]{CGRNFIOJMPA2013} and \cite[Theorem 3.2]{grochenig2}.
\end{proof}

The proof of the characterization of  the  symbol classes $S^m(\rdd)$   claimed in Theorem \ref{Th3.3}, can be inferred easily from the result above.

\begin{proof}[Proof of Theorem \ref{Th3.3}] 
	The proof is a direct application of the characterization of the classes $S^m(\rdd)$ presented in \eqref{e2}  and Theorem \ref{teor41}.
\end{proof}

The following issue is an improvement of \cite[Theorem 2.4]{ElenaHermite} and relies on the new characterization of $S^m(\rdd)$ proved in Lemma \ref{charsm00}.
\begin{proposition}\label{Th.3.4}
	Consider  $g\in \cS(\rd)\setminus\{0\}$, $m\in\bR$ and $\sigma\in S^m\left(\mathbb{R}^{2d}\right)$. For any $n\in\bN$ there exists $C=C(n)>0$, which does not depend on $\sigma$ or $\tau$, such that
	\begin{equation}\label{Ineq-Th.3.4}
	\left|\left\langle \opt \left(\sigma\right)\pi\left(z\right)g,\pi\left(u\right)g\right\rangle \right|\le C |\sigma|_{n,m}\frac{\la\mathcal{T}_{\tau}(z,u)\ra^m}{\la u-z\ra^n }, \qquad\forall\tau\in[0,1],\,\,\,\forall u,z\in\mathbb{R}^{2d}.
	\end{equation}
\end{proposition}

\begin{proof} Using the characterization of the H\"{o}rmander classes $S^m(\rdd)$ in \eqref{e2} we infer that $\sigma\in M^{\infty}_{\la\cdot\ra^{-m}\otimes\la\cdot\ra^n}(\rdd)$ and, for any $n\in\bN$, the norm estimate in \eqref{e1} says that there exists  $C=C(n,m)$ such that
	\begin{equation}\label{e7}
	\norm{\sigma}_{M^{\infty}_{\la\cdot\ra^{-m}\otimes\la\cdot\ra^n}}\leq C(n,m) |\sigma|_{n,m},
	\end{equation}
	where $C(n,m)>0$ is independent of $\sigma$.	For $z,w\in\rdd$ we use Lemma \ref{lem:STFT-gaborm} and  the  norm estimate in \eqref{e7} which yield
	\begin{align*}
		\left|\left\langle \opt \left(\sigma\right)\pi\left(z\right)g,\pi\left(u\right)g\right\rangle \right|&=\left|{V}_{\Phi_{\tau}}\sigma\left(\mathcal{T}_{\tau}\left(z,u\right),J\left(u-z\right)\right)\right|\\
		&\leq C 	|\sigma|_{n,m}\frac{\la\mathcal{T}_{\tau}(z,u)\ra^m}{\la u-z\ra^n },
	\end{align*} 
that is the  desired result.
\end{proof}

For $s\in[0,+\infty)\setminus\bN$, the estimate reads as follows.
\begin{proposition}
	Consider  $g\in \cS(\rd)\setminus\{0\}$, $\tau\in\left[0,1\right]$, $m\in\bR$ and $\sigma\in S^m\left(\mathbb{R}^{2d}\right)$. For any $s\in[0,+\infty)\setminus\bN$ there exists $C=C(s,m)>0$, which does not depend on $\sigma$ or $\tau$, such that
	\begin{equation}\label{Eq-46}
	\left|\left\langle \opt \left(\sigma\right)\pi\left(z\right)g,\pi\left(u\right)g\right\rangle \right|\le C |\sigma|_{n+1,m}\frac{\la\mathcal{T}_{\tau}(z,u)\ra^m}{\la u-z\ra^s }, \qquad\forall u,z\in\mathbb{R}^{2d},
	\end{equation}
	where $n=[s]$ is the integer part of $s$.
\end{proposition}
\begin{proof}
The result is attained by the  the same argument as Proposition \ref{Th.3.4} and the inclusion relations between modulation spaces in \eqref{inclmod}.
\end{proof}

\subsection{Boundedness results}\label{3.1}
The characterization of the class $S^m$ in Lemma \ref{charsm00} and Theorem \ref{3.1} are the key tool for  boundedness properties of $\tau$-operators on weighted modulation spaces.  
\begin{proposition}\label{Proposition-continuity-Shubin}
	Consider $\tau\in\left[0,1\right]$, $m\in\bR$, $\sigma\in S^m(\rdd)$, $0<p,q\leq\infty$. Then $\opt(\sigma)$, from $\cS(\rd)$ to $\cS'(\rd)$,  extends uniquely to a bounded operator  
	\begin{equation*}
		\opt(\sigma)\colon M^{p,q}_{\la\cdot\ra^{r+m}}(\rd)\to M^{p,q}_{\la\cdot\ra^{r}}(\rd),
	\end{equation*}
for every $r\in\bR$.
\end{proposition}
\begin{proof}
		Choose $g\in\cS(\rd)$ and a lattice $\Lambda$ such that $\cG(g,\Lambda)$ is a Gabor frame for $\lrd$. Define $t:=\min\{1,p,q\}$ and choose $s> (2d+|r|)/t$. Using the equivalent discrete (quasi-)norm for the modulation space, see e.g. \cite[Proposition 1.5]{ToftquasiBanach2017}, the estimate in \eqref{eq:almdiag discr} and Young's convolution inequality in \cite[Theorem 3.1]{Galperin2014}, we obtain the result. Namely,
	\begin{align*}
	\norm{\opt(\sigma)f}_{M^{p,q}_{\la\cdot\ra^{r}}}&\asymp\norm{V_g (\opt(\sigma)f)}_{\ell^{p,q}_{\la\cdot\ra^{r}}(\Lambda)}
	\leq \left\|h_\tau\ast\abs{V_g f}\la\cdot\ra^{\abs{m}}\right\|_{\ell^{p,q}_{\la\cdot\ra^{r}}(\Lambda)}\\
	&\leq  
	\left\|h_\tau\right\|_{\ell^t_{\la\cdot\ra^{s}(\Lambda)}}\left\|V_g f\la\cdot\ra^{m}\right\|_{\ell^{p,q}_{\la\cdot\ra^{r}(\Lambda)}}\leq  C \left\| f\right\|_{M^{p,q}_{\la\cdot\ra^{r+m}}}.
	\end{align*}	
	Alternatively, since $\sigma \in S^m= \bigcap_{s\geq 0} M^{\infty,q}_{\la \cdot\ra^{-m}\otimes \la \cdot\ra^s}(\rdd)$ by Lemma \ref{charsm00}, one can  use \cite[Theorem 3.1]{ToftquasiBanach2017} with $p=\infty$ and $q\leq1$ small enough to yield the claim.
\end{proof}

\begin{remark}
	(i) For  $\sigma \in S^0(\rdd)=S^0_{0,0}(\rdd)$ and we recapture the continuity of
	\begin{equation*}
	\opt(\sigma)\colon M^{p,q}_{\la\cdot\ra^{r}}(\rd)\to M^{p,q}_{\la\cdot\ra^{r}}(\rd).
	\end{equation*}
	This was already shown in \cite{Toftweight2004} for $p,q\geq 1$, for the quasi-Banach cases see \cite{ToftquasiBanach2017}.\\
	(ii) For $p=q=2$ we have the continuity between the Shubin-Sobolev spaces $Q_{r+m}(\rd)$ and $Q_{r}(\rd)$.\\
\end{remark}

\begin{corollary}
	Consider $\tau\in\left[0,1\right]$, $m,r\in\bR$, $\sigma\in S^m(\rdd)$, $0<p,q\leq\infty$. Let $\norm{\opt(\sigma)}$ denote the norm of $\opt(\sigma)$ in $B(M^{p,q}_{\la\cdot\ra^{r+m}}(\rd),M^{p,q}_{\la\cdot\ra^{r}}(\rd))$. Then there exists a constant $C>0$ such that
	\begin{equation}
		\norm{\opt(\sigma)}\leq C,\qquad\forall\tau\in[0,1].
	\end{equation}
\end{corollary}
\begin{proof}
	The claim is evident from proof of Proposition \ref{Proposition-continuity-Shubin}.
\end{proof}

\subsection{Born-Jordan operators}
The Born-Jordan operator with symbol $\sigma\in\cS'(\rd)$ can be defined as
$$\la \opj(\sigma) f,g\ra =\la \sigma, W_{BJ}(g,f)\ra, \quad f,g\in\cS(\rd),$$
where the Born-Jordan distribution $ W_{BJ}(g,f)$ is
$$ W_{BJ}(g,f)=\int_0^1 W_\tau(g,f) \,d\tau,$$
see, e.g., the textbook \cite{deGossonBJ}. In what follows we study the Gabor matrix decay for Born-Jordan operators.
\begin{theorem}\label{teor41B}
	Consider  $g\in \cS(\rd)\setminus\{0\}$. For $m\in\bR$ consider $\sigma\in S^m\left(\mathbb{R}^{2d}\right)$. Then for every $s\geq 0$, $0<q\leq\infty$, $\tau\in [0,1]$ there exists
	a function $H_\tau \in L^q_{\la \cdot\ra^{s}}(\rdd)$ which satisfies \eqref{H-tau} and such that
		\begin{equation}
		\left|\left\langle \opj \left(\sigma\right)\pi\left(z\right)g,\pi\left(u\right)g\right\rangle \right|\le \la z \ra^m \int_0^1H_\tau(u-z)\, d\tau, \qquad\forall u,z\in\mathbb{R}^{2d}.\label{eq:almdiagBJ}
		\end{equation}
\end{theorem}
\begin{proof}
	For $\sigma\in\cS'(\rdd)$, $\opj(\sigma)$ is linear and continuous from $\cS(\rd)$ into $\cS'(\rd)$, see \cite{deGossonToft-BJ2020}.  For $z,u\in\rdd$, $\sigma\in S^m(\rdd)$ and $g\in\cS(\rd)$ we compute
	\begin{align*}
		\la\opj(\sigma)\pi(z)g,\pi(u)g\ra&=\la\sigma,W_{BJ}(\pi(u)g,\pi(z)g)\ra\\
		&=\int_\rdd\sigma(y)\int_0^1\overline{W_{\tau}(\pi(u)g,\pi(z)g)(y)}\,d\tau dy\eqqcolon I.
	\end{align*}
	From \cite[Proposition 2.2, Remark 2.3]{deGossonToft-BJ2020} we have that the mapping
	\begin{equation*}
		\bR\times\cS(\rd)\times\cS(\rd)\to\cS(\rdd),\quad(t,\varphi,\psi)\mapsto W_t(\varphi,\psi)
	\end{equation*} 
	is continuous and locally uniformly bounded. Thus $W_{BJ}(\varphi,\psi)\in\cS(\rdd)$ and the integral $I$ is absolutely convergent, so that
	\begin{equation*}
		I=\int^1_0\int_\rdd \sigma(y)\overline{W_{\tau}(\pi(u)g,\pi(z)g)(y)}\,dyd\tau=\int^1_0\left\langle \opt \left(\sigma\right)\pi\left(z\right)g,\pi\left(u\right)g\right\rangle\,d\tau.
	\end{equation*}
By Peetre's inequality:
	\begin{align*}
	\la\mathcal{T}_{\tau}(z,u)\ra^m&=\la z_1+\tau(u_1-z_1),z_2+(1-\tau)(u_2-z_2)\ra^m\\
	&\lesssim\la z\ra^m\la u-z\ra^{|m|},
	\end{align*} 
	for every $u=(u_1,u_2), \,z=(z_1,z_2)\in\rdd$. Hence, 	using Theorem \ref{Th3.3},
	\begin{equation*}
		\abs{I}\leq\int^1_0\abs{\left\langle \opt \left(\sigma\right)\pi\left(z\right)g,\pi\left(u\right)g\right\rangle}\,d\tau\lesssim \int^1_0 H_\tau(u-z)\,\la u-z\ra^{|m|}  \,d\tau\,\la z\ra^m.
	\end{equation*} 
Then the function $H_\tau(z) \,\la z\ra^{|m|}$ satisfies  condition \eqref{H-tau}.
\end{proof}
\begin{remark}
	(i) For $q\geq 1$, we can define $H(z):=\int_0^1 H_\tau(z)d\tau$.  Using Minkowski's integral inequality we infer $H\in L^q_{\la\cdot\ra^s}(\rdd)$ and the estimate \eqref{eq:almdiagBJ} becomes 
		\begin{equation*}
	\left|\left\langle \opj \left(\sigma\right)\pi\left(z\right)g,\pi\left(u\right)g\right\rangle \right|\le  H (u-z)\, \la z \ra^m, \qquad\forall u,z\in\mathbb{R}^{2d}.
	\end{equation*}
 Notice that for $0<q<1$ Minkowski's integral inequality is not true in general. \par \noindent
(ii) Arguing as in Theorem \ref{teor41B}, we may discretize the Gabor matrix decay in \eqref{eq:almdiagBJ} as follows:   consider $g\in \cS(\rd)\setminus\{0\}$ and a lattice $\Lambda $ in $\rdd$ such that  $\mathcal{G}\left(g,\Lambda\right)$ is a  Gabor frame for $L^{2}\left(\mathbb{R}^{d}\right)$. If $\sigma\in S^m\left(\mathbb{R}^{2d}\right)$  then for every $s\geq 0$, $0<q\leq\infty$,  there exists
a sequence $h_\tau\in \ell^q_{\la \cdot \ra^s}(\Lambda)$ with $\|h_\tau \|_{\ell^q_{\la \cdot \ra^s}}\leq C$ for every $\tau \in [0,1]$ such that
\begin{equation*}
\left|\left\langle \opj \left(\sigma\right)\pi\left(\mu\right)g,\pi\left(\lambda\right)g\right\rangle \right|\leq\la\mu\ra^m\int_0^1 h_\tau(\lambda-\mu)d\tau, \qquad\forall\lambda,\mu\in\Lambda.
\end{equation*}
\end{remark}

\begin{corollary}\label{Cor-Continuity-BJ}
	Consider $m\in\bR$, $\sigma\in S^m(\rdd)$, $0<p,q\leq\infty$. Then $\opj(\sigma)$, from $\cS(\rd)$ to $\cS'(\rd)$,  extends uniquely to a bounded operator  
	\begin{equation*}
	\opj(\sigma)\colon M^{p,q}_{\la\cdot\ra^{r+m}}(\rd)\to M^{p,q}_{\la\cdot\ra^{r}}(\rd),
	\end{equation*}
	for every $r\in\bR$.
\end{corollary}
\begin{proof}
	The proof is similar to the one of Proposition \ref{Proposition-continuity-Shubin}, using the decay for Gabor matrix of $\opj(\sigma)$ found in Theorem \ref{teor41B}, with $h_\tau$ replaced by $\int_0^1 h_\tau(\cdot)d\tau$. Then, for $t\geq 1$ we use Minkowski's inequality to write $$\left\|\int_0^1 h_\tau(\cdot)d\tau\right\|_{\ell^t_{\la \cdot\ra^s}}\leq \int_0^1 \| h_\tau\|_{\ell^t_{\la \cdot\ra^s}}d\tau \leq C.$$
	 For $t<1$  we use  the inclusion relations \eqref{inclusion} and majorize  
	$$\left\|\int_0^1 h_\tau(\cdot)d\tau\right\|_{\ell^t_{\la \cdot\ra^s}}\lesssim \left\|\int_0^1 h_\tau(\cdot)d\tau\right\|_{\ell^1_{\la \cdot\ra^{\tilde{s}}}},$$
	with $\tilde{s}\geq 0$ such that $1/t+s/(2d)<1+\tilde{s}/(2d)$, that is $$\tilde{s}>\frac{2d}{t}(1-t),$$
	and we proceed as above.
\end{proof}

\section*{Acknowledgements}
The authors would like to thank Fabio
Nicola and S. Ivan Trapasso  for fruitful conversations and comments. The authors are very grateful to the reviewers for their comments and in particular for the improvements of Proposition \ref{Proposition-continuity-Shubin}. \par 
The first author was partially supported by MIUR grant Dipartimenti di Eccellenza 20182022, CUP: E11G18000350001, DISMA, Politecnico di Torino.

\end{document}